\documentclass[11pt]{amsart}
\usepackage{amsaddr}

\usepackage{graphicx}

\usepackage{epsfig}
\usepackage{mathtools, esdiff}
\usepackage[subnum]{cases}
\usepackage{enumerate}
\usepackage{amsmath,amssymb, amsfonts, amsthm}
\usepackage{color,physics}
\usepackage[style=numeric-comp,backend=biber, url=false,doi=false, bibencoding=utf8,giveninits=true,maxnames=50]{biblatex}
\DeclareNameAlias{default}{family-given}
\bibliography{references}
\AtEveryBibitem{\clearfield{issn}}
 \usepackage[bookmarks=true,
bookmarksnumbered=true, breaklinks=true,
pdfstartview=FitH, hyperfigures=false,
plainpages=false, naturalnames=true,
colorlinks=true,
pdfpagelabels]{hyperref}

\renewbibmacro{in:}{%
  \ifentrytype{article}
    {}
    {\bibstring{in}%
     \printunit{\intitlepunct}}}

\newcommand{\R}{\mathbb{R}}
\newcommand{\Md}[1]{\mathcal{F}(#1)}

\newcommand{\s}{\mathbb{S}}

\newcommand{\conbod}{\mathcal{K}^n}

\newcommand{\vol}{\text{\rm Vol}}

\newtheorem{theorem}{Theorem}[section]

\newtheorem{lemma}[theorem]{Lemma}
\newtheorem{corollary}[theorem]{Corollary}

\title[Variational Logarithmic Minkowski Problem]{The (Self-Similar, Variational) Rolling Stones}
\author{Dylan Langharst and Jacopo Ulivelli}
\thanks{
The first named author was supported in part by the U.S. National Science Foundation Grant DMS-2000304 and the Chateaubriand Fellowship by the French embassy in the United States. This work was completed while the first named author was a postdoctoral researcher funded by a Fondation Sciences Math\'ematiques de Paris fellowship. Both authors were supported by the National Science Foundation under Grant DMS-1929284 while in residence at the Institute for Computational and Experimental Research in Mathematics in Providence, RI, during the Harmonic Analysis and Convexity program.
}

\begin{document}

\maketitle
\begin{abstract}The interplay between variational functionals and the Brunn-Minkowski Theory is a well-established phenomenon widely investigated in the last thirty years. In this work, we prove the existence of solutions to the even logarithmic Minkowski problems arising from variational functionals, such as the first eigenvalue of the Laplacian and the torsional rigidity. In particular, we lay down a blueprint showing that the same result holds for more generic functionals by adapting the volume case from B\"or\"oczky, Lutwak, Yang, and Zhang. We show how these results imply the existence of self-similar solutions to variational flow problems \`a la Firey's worn stone problem.
\end{abstract}
% \newpage 
\section{Introduction}
The starting point of this work, like many works as of late, is Minkowski's existence theorem. Let us describe the main idea. We recall that if $K$ is a convex body (i.e., a compact, convex set with non-empty interior) in the $n$-dimensional Euclidean space $\R^n$, then its \textit{surface area measure} $S_K$ is a Borel measure on the unit sphere $\s^{n-1}$ given by the following: For every Borel set $A \subset \s^{n-1},$ $$S_K(A)=\mathcal{H}^{n-1}(N^{-1}_K(A)),$$ where $\mathcal{H}^{n-1}$ is the $(n-1)$-dimensional Hausdorff measure and $N_K:\partial K \rightarrow \s^{n-1}$ is the Gauss map, which associates an element $y$ of $\partial K$ (the boundary of  $K$) with its outer unit normals. 

Given a finite Borel measure $\mu$ on $\s^{n-1}$, one may ask: Does there exist a unique (up to translations) convex body $K$ such that $dS_K=d\mu$? Minkowski's existence theorem \cite[Section 8.2]{Sh1} shows that if $\mu$ satisfies the following two conditions, then the answer is yes:
\begin{enumerate}
    \item The measure $\mu$ is not concentrated on any great hemisphere, that is
    \label{con_1}
    $$\int_{\s^{n-1}}|\langle \theta, \xi \rangle|d\mu(\xi) > 0 \quad \text{for all } \theta\in\s^{n-1}.$$
    \item The measure is centered, that is
    $$\int_{\s^{n-1}}\xi d\mu(\xi)=0.$$
    \label{con_2}
\end{enumerate}
The Brunn-Minkowski theory in convex geometry has received many generalizations, see, e.g., \cite{LE98,LE93,LE96,LYZ10,DJ96_2,HLYZ16,BLYZ12}. The generalizations usually have a corresponding non-trivial version of Minkowski's existence theorem, see, e.g., \cite{CW06,LYZ04,LYZ06,CF10,HLYZ10,HXZ21, KL23,GAL19, BLYZ13}. Such generalizations are traditionally called \textit{Minkowski Problems}. In this paper, we focus on one Minkowski problem in particular: The even logarithmic Minkowski problem. The existence of a solution to this problem was shown by B\"or\"oczky, Lutwak, Yang, and Zhang \cite{BLYZ13}. We adapt their proof, showing that the ideas established therein have far-reaching consequences and can be implemented for various problems. 

We will start with some basic definitions, and the textbook by Schneider \cite{Sh1} serves as a general reference. Let $\conbod$ denote the class of convex bodies containing the origin in their interior, and let $\conbod_e$ denote the class of symmetric convex bodies, where $K$ is symmetric if $K=-K.$ Recall that $K\in\conbod$ is uniquely determined by its support function, given by $h_K(x)=\sup_{y\in K}\langle y,x \rangle.$ We remind the reader that the \textit{cone measure} of a convex body $K$ is $\frac{1}{n}h_K(u)dS_{K}(u),$ and that the cone measure satisfies 
$$\vol_n(K)=\frac{1}{n}\int_{\s^{n-1}}h_K(u)dS_K(u),$$
where $\vol_n(K)$ is the $n$-dimensional volume (Lebesgue measure) of $K$. In \cite{BLYZ13}, Böröczky, Lutwak, Yang, and Zhang showed that given an even, non-concentrated Borel measure $\nu$ on $\s^{n-1},$ there exists some symmetric $K\in\conbod_e$ such that $$d\nu(u)=\frac{1}{n}h_K(u)dS_K(u)$$ if, and only if, $\nu$ satisfies the \textit{strict subspace concentration condition} i.e., for every subspace $H$ of $\R^n,$ $0< \text{dim } H < n,$ one has \begin{equation}\nu\left(H\cap \s^{n-1}\right)< \frac{1}{n}\nu\left(\s^{n-1}\right)\text{dim}(H).
\label{eq:SSCC}
\end{equation}
This is known as the even logarithmic Minkowski problem. The uniqueness remains open, and it would imply the famed logarithmic-Brunn-Minkowski inequality (see, e.g., \cite{BLYZ12,BLYZ13}). The subspace concentration condition appeared earlier in the literature, dating back at least to \cite{BK10}.

The logarithmic Minkowski problem is related to the fate of worn stones shown by Firey \cite{WF74}, specifically, that they become spherical. We recall a convex body $K$ is $C^2_+$ if it has strictly positive Gauss curvature, and at every $x\in \partial K$, there exists a unique tangent plane to $\partial K$. We note that in this instance, the Gauss map $N_K$ is a diffeomorphism between $\partial K$ and $\s^{n-1}$. We denote by $\kappa_K$ the Gauss curvature of a $C^2_+$ body $K$. Let $\{K(t)\}$ be a collection of convex bodies of class $C^2_+$. We write $h(t,\xi):=h_{K_{t}}(\xi)$ and $\kappa(t,\xi)=\kappa_{K(t)}(N^{-1}_K(\xi))$, where $\xi\in\s^{n-1}$. Firey explained how a worn stone can be modeled through the following PDE: For $\xi\in\s^{n-1}$ and $t\in [0, T)$ with fixed constant $T>0$
\begin{equation}
\label{eq:WornPDE}
    \pdv{h(t,\xi)}{t}=-(T\varphi)\kappa(t,\xi).
\end{equation}
%a rescaling for $t \in [0,1)$ might be a better option
for some constant of proportionality $\varphi>0$. The initial data $h(0,\xi)$ is the support function of the smooth convex body $K(0)$ that is being worn down through an abrasion process. We set $h_{K(0)}(\xi)=h(0,\xi)$. We note that we are following the problem as stated by Tso \cite{KT85} and Andrews \cite{BA99}, who showed that a solution exists. This can be equivalently stated as the following hyperbolic Monge-Amp\`ere equation:
\[
\pdv{h(t,\xi)}{t} \det(D^2h(t,\xi)+h(t,\xi)\text{I}_{n-1})=-T\varphi,
\]
where $\text{I}_{n-1}$ is the $(n-1)\times (n-1)$ identity matrix and $D^2$ is the spherical Hessian. 

We say a solution to the above Monge-Amp\`ere equation is \textit{self-similar with death time $T$} if $h(t,\xi)=T^{-\frac{1}{n}}h_{K(0)}(\xi)(T-t)^{\frac{1}{n}}$ for some $T>0$. Then, we obtain $\pdv{h(t,\xi)}{t}=-T^{-\frac{1}{n}}\frac{1}{n}h_{K(0)}(\xi)(T-t)^{\frac{1-n}{n}}$. From the homogeneity of the determinant and support function, we obtain
\[
\frac{1}{n}h_{K(0)}(\xi) \det(D^2h_{K(0)}(\xi)+h_{K(0)}(\xi)\text{I}_{n-1})=\varphi.
\]
It is well known \cite{Sch70} that for $C^2_+$ bodies, one has
$$\det(D^2h_{K(0)}(\xi)+h_{K(0)}(\xi)\text{I}_{n-1}) = \frac{dS_{K(0)}(\xi)}{d\xi}.$$
Consequently, self-similar solutions to the worn stone PDE (\eqref{eq:WornPDE} before) imply the Minkowski problem 
\[
\frac{1}{n}h_{K(0)}(\xi)  dS_{K(0)}(\xi)=\varphi d\xi,
\]
which is the same as the log-Minkowski problem for a multiple of the spherical Lebesgue measure. Hence, self-similar solutions exist via the work of Böröczky, Lutwak, Yang, and Zhang \cite{BLYZ13}. Notice, as $t\to T^-$, $K(t)$ approaches a singleton containing the origin. We remark that the original worn stone problem considered by Firey \cite{WF74} had, in place of the constant $\varphi$, $\vol_n(K(t))$. Furthermore, he proved that if a solution exists, then $K(t)$ goes to a centered Euclidean ball as $t\to\infty$. However, from the homogeneity of the volume, one can verify that a self-similar solution of the form $h(t,\xi)=T^{-\beta}h_{K(0)}(\xi)(T-t)^\beta$ cannot exist in Firey's version of the flow via direct substitution. 

The main goal of this work is, given a class of functionals arising from the calculus of variations, to establish the existence of solutions for the corresponding even logarithmic Minkowski problems. The proof adapts the volume case from \cite{BLYZ13}. The textbook by Evans \cite{Evans} will serve as a reference for many of these facts. In this introduction, we focus on one such variation in particular: Torsional rigidity. Recall that for a bounded Borel set $\Omega\subset\R^n$ with some regularity assumptions (for our purposes, convex suffices), the torsional rigidity of $\Omega$ is defined as
\[\tau(\Omega)\coloneqq \sup \left\{\left(\int_{\Omega} |w(x)|\, dx \right)^2\left(\int_{\Omega} |\nabla w(x)|^2\, dx\right)^{-1} : w \in H^1_0(\Omega), u \neq 0 \right\},\] where $H^1_0(\Omega)$ is as usual the Sobolev space obtained as the closure in $L^2(\Omega)$ of the set of absolutely continuous functions with compact support in $\Omega$. It is a standard fact that it is possible to write \[\tau(\Omega)=\int_\Omega |\nabla u_\Omega (x)|^2\, dx,\]
where $u_\Omega$ is the solution to the boundary value problem
\begin{equation}\begin{cases}
-\Delta u(x) = 1 &\text{for } x \in \Omega, \\
\quad u(x)=0 &\text{for } x \in \partial \Omega.
\end{cases}
\label{eq:tor_PDE}
\end{equation}
Here, $\Delta$ denotes the standard Laplacian on $\R^n$. Standard results in the theory of elliptic equations (see, for example, \cite{GT01}) guarantee existence and uniqueness for the solution of \eqref{eq:tor_PDE}, which is always of class $C^\infty (\Omega) \cap C(\overline{\Omega})$, where $\overline{\Omega}$ denotes the closure of $\Omega$. It was proved by Dahlberg \cite{BD77} that $\nabla u_\Omega$ exists $\mathcal{H}^{n-1}-$almost everywhere on $\partial \Omega$. See \cite[Section 2]{CF10} for more details on the properties of $\nabla u_\Omega$.

A further representation of $\tau(\Omega)$, better suited for our investigation, can be obtained  for $K\in\conbod$ and $\Omega=\text{int}(K)$ via the \textit{Hadamard formula} \cite[Theorem 3.1]{CF10} as an integral over $\partial K$; in an abuse of notation, we will identify $\Omega$ with $K$:
\begin{equation}\tau\left(K\right)=\frac{1}{n+2}\int_{\partial K}h_K(N_K(x))|\nabla u_K(x)|^2d\mathcal{H}^{n-1}(x).
\label{eq:had_tor}
\end{equation}
Considering the pushforward through the Gauss map $N_K$, we can then define on $\mathbb{S}^{n-1}$ the \textit{torsional measure} 
\begin{equation}\label{eq:tor_meas}
    \mathcal{T}_K(E)\coloneqq \frac{1}{n+2}\int_{N_K^{-1}(E)} h_K(N_K(x))|\nabla u_K(x)|^2\, d\mathcal{H}^{n-1}(x)
\end{equation}
for every Borel set $E \subset \mathbb{S}^{n-1}$. This measure will be the prototype of our investigation in this work. We remark that this very measure was at the center of the recent work of  Crasta and Fragal\'a \cite{CF23}, where they proved versions of Firey's result on the fate of a rolling stone for measures arising from the calculus of variations, including torsional rigidity. 

In this work, we attempt to streamline the approach of \cite{BLYZ13} for measures as \eqref{eq:tor_meas}. In particular, our conclusions will also apply to the first eigenvalue of the Laplacian. As a consequence of our main result (see Theorem~\ref{mainTh} below), we obtain sufficient conditions for the existence of a solution to the torsional log-Minkowski problem.
\begin{theorem}
\label{t:tor_prob}
    Let $\nu$ be an even, finite Borel measure on $\s^{n-1}$ such that $\nu$ satisfies the subspace concentration condition \eqref{eq:SSCC}. Then, there exists an origin symmetric convex body $K\in\conbod_e$ such that \[\nu=\mathcal{T}_K.\] In particular, one has
    $$\nu\left(\s^{n-1}\right)=\tau\left(K\right).$$
\end{theorem}

We note that necessary conditions remain an interesting open problem to settle. For the volume case in \cite{BLYZ13}, this was possible through a smart use of the subspace concentration condition, which we were not able to replicate. Indeed, the argument in \cite{BLYZ13} relies on an interaction between \eqref{eq:SSCC} and Minkowski addition. This is no longer the case for torsional rigidity: For example, the infimal convolution (i.e., Minkowski addition of epigraphs) of two solutions in dimension $n$ (appropriately embedded in dimension $2n$) does not provide a solution in dimension $2n$. This is easily verified considering solutions on balls. Moreover, while \eqref{eq:SSCC} holds trivially for convex bodies with smooth boundary, in the polytopal case it is extremely difficult to provide examples since explicit solutions of \eqref{eq:tor_PDE} are not accessible. This logistic constraint can be avoided in some particularly regular instances, such as centrally symmetric fair dices (see, for example, \cite{DK89}). A polytope $P \in \conbod$ is a fair dice if its group of symmetries acts transitively on the set of facets. In this case, the cone-volume density in \eqref{eq_had_push} can be uniformly scaled, and the argument in \cite{BLYZ13} works verbatim. 

Our exposition is structured as follows. In Section~\ref{sec:diagonal}, we introduce the more general framework of our investigation; we define a general type of functional, which we call a set-dependent, $\alpha$-homogeneous Borel measure, of which $\mathcal{T}_K$ is an example. Section~\ref{sec:main} is dedicated to proving Theorem~\ref{mainTh}, the main result of this paper, where we show the existence of solutions for the even logarithmic Minkowski problems of set-dependent, $\alpha$-homogeneous Borel measures $\mathcal{F}$ satisfying some structural assumptions, which we call properties $\eqref{A},\eqref{B}$, and $\eqref{C}$, on $\conbod_e$. Theorem \ref{t:tor_prob} will follow from this result, as well as an analogous statement providing the existence of solutions to the even logarithmic Minkowski problem of the first eigenvalue of the Laplacian, Theorem~\ref{t:first_eigen}. See Section~\ref{sec:diagonal} for properties of the first eigenvalue. In Section~\ref{sec:worn}, we explain how our results prove the existence of self-similar solutions to a variational version of the Firey-Tso \cite{WF74,KT85} worn stone problem studied recently by Crasta and Fragal\'a \cite{CF23}. We conclude with a discussion on capacity in Section \ref{sec:remarks_on_capacity}.

\section{A General Perspective}
\label{sec:gen}
\subsection{Weighted Brunn-Minkowski Theory}
Our first step is deeply intertwined with the weighted Brunn-Minkowski theory. See \cite{KL23,FLMZ23_1} and the references therein for a quick overview.

Let us now recall how to represent integration over the boundary of a compact, convex set as integration over the sphere via the Gauss map.
For a convex body $K\in\conbod$ and a Borel measure $\mu$ \textit{on the boundary of $K$} with density $\phi$, the weighted surface area of $K$ with respect to $\mu$ is defined by
\begin{equation}
    \label{eq:surface_mu}
    S^{\mu}_K(E)=\int_{N_K^{-1}(E)}\phi(x)d\mathcal{H}^{n-1}(x)
\end{equation}
for every Borel set $E \subset \s^{n-1}.$ Let $\Lambda$ be the set of all locally finite, regular Borel measures $\mu$ with non-negative Radon-Nikodym derivative, i.e., density:
\[
\mu \in \Lambda \iff \frac{d\mu(x)}{dx} = \phi(x), \text{ with } \phi \colon \R^n \to \R^+,\phi\in L^1_{\text{loc}}(\R^n). 
\]
We now extend the weighted surface area measure to Borel measures $\mu\in\Lambda.$ Let $B_2^n$ denote the Euclidean unit ball. Let $K\subset\R^n$ be a Borel set. Recall the classical definition of the Minkowski content of $K$ with respect to a Borel measure $\mu$ is given by
$$\mu^+(\partial K)\coloneqq\liminf_{\epsilon\to 0}\frac{\mu\left(K+\epsilon B_2^n\right)-\mu(K)}{\epsilon}.$$
It was folklore for quite some time (see, e.g. the work by K. Ball concerning the Gaussian measure \cite{Ball93}) that if $\mu\in\Lambda$ has continuous density and $K\in\conbod$, then the $\liminf$ is a limit and
\begin{equation}\mu^+(\partial K)=\lim_{\epsilon\to 0}\frac{\mu\left(K+\epsilon B_2^n\right)-\mu(K)}{\epsilon}=\int_{\partial K}\phi(x)d\mathcal{H}^{n-1}(x).
\label{eq_bd}
\end{equation}
This was shown rigorously by Livshyts \cite{GAL19}. More recently, it was shown by the first named author and Kryvonos \cite{KL23} that \eqref{eq_bd} holds for every $\mu\in\Lambda$ when $K\in\conbod$ (i.e. that the $\liminf$ is a limit and the integral formula holds), as long as $\phi$ contains $\partial K$ in its Lebesgue set. Under this minor assumption on $\mu$, we can with justification refer to $\mu^+(K)$ as the weighted surface area of $K$. Next, by the pushforward through $N_K$, we then define on $\s^{n-1}$ a corresponding Borel measure $S^{\mu}_K$ via \eqref{eq:surface_mu} to get $\mu^+(\partial K) = S^\mu_K(\s^{n-1})$. Livshyts \cite{GAL19} first did the even Minkowski problem for $S^\mu_K$ when $\mu$ is $\alpha$-homogeneous, $1/\alpha$-concave, $\alpha \geq n$ with density. Later, Huang, Xi, and Zhao \cite{HXZ21} did the even Minkowski problem when $\mu$ is the Gaussian measure; Langharst and Kryvonos \cite{KL23} followed this development by solving the even Minkowski problem when $\mu$ is any radially decreasing probability measure with density.

\subsection{Set-dependent Borel measures and their Minkowski problems}
\label{sec:set_dep}
We now view $\conbod$ as a space closed under the Hausdorff topology, and we view $\Lambda$ as a space equipped with the $L^1_{\text{loc}}$ topology. Assume now that, instead of a fixed measure $\mu \in \Lambda$, we have a continuous map  
\begin{equation}\label{eq:set_dep_meas}
    \begin{split}
         \conbod &\to \Lambda\\
    K &\mapsto \mu_K.
    \end{split}
\end{equation}
We will denote the density of $\mu_K$ as $\phi_K$. Heuristically, $\conbod$ is viewed as an indexing set for a collection of Borel measures. If $\mathcal{B}^n$ denotes the collection of Borel sets on $\R^n$, then we call the operator $\conbod\times \mathcal{B}^n\mapsto \R^+$ given by $(K,E)\mapsto\mu_K(E)$ for every $E\in\mathcal{B}^n$ and $K\in\conbod$ a \textit{set-dependent Borel measure}. Let us consider two examples. If we work with a constant map, i.e., there exists $\mu\in\Lambda$ such that $\mu_K=\mu$ for every $K\in\conbod,$ then $\mu_{K}(E)=\mu(E)$. 

A non-constant example is given instead by torsional rigidity: Let the open set $\Omega$ from \eqref{eq:tor_PDE} be $\text{int}(K)$ for a $K\in\conbod,$ and let $u_K$ solve \eqref{eq:tor_PDE} for this $K$. Then, we can define for every Borel set $E \subset \R^n$
\begin{equation}\label{eq:corr_tor}
    \mu_K(E)=\int_E |\nabla u_K(x)|^2\, dx.
\end{equation}
Notice, then, that $\tau(K)=\mu_K(K)$. Continuity of the map $K\mapsto \mu_K$ in this instance is far from trivial; see, e.g., \cite{CF10} for the torsional rigidity, \cite{DJ96_2} for the capacity, and \cite{CNSXYZ15,CS03} for the $p$-capacity. 

Anyway, returning to the generic case of a continuous map $K\mapsto \mu_K$ in \eqref{eq:set_dep_meas}: For every fixed $K$, we apply \eqref{eq_bd} to the $\mu_K$-measure of the body $K$ itself to obtain:
$$\mu_K^+(\partial K) = \int_{\partial K}\phi_K(x)\mathcal{H}^{n-1}(x)=\int_{\s^{n-1}}dS_K^{\mu_K}(u),$$
where the final equality follows from the Gauss map, i.e., $S_K^{\mu_K}$ is defined via \eqref{eq:surface_mu}:
\begin{equation}
    \label{eq:surface_mu_2}
    S^{\mu_K}_K(E)=\int_{N_K^{-1}(E)}\phi_K(x)d\mathcal{H}^{n-1}(x).
\end{equation}
We define a \textit{diagonal set-valued Borel measure} as $$\Md{K}=\mu_K(K).$$ 
The \textit{Minkowski problem of the functional $\mathcal{F}$} is therefore answering the question: Given a Borel measure $\nu$ on $\s^{n-1}$ (with some reasonable restrictions) and a functional $\mathcal{F}$, does there exist a convex body $K$ such that $d\nu=dS^{\mu_K}_K$? The Minkowski problem for torsional rigidity was solved by Colesanti and Fimiani \cite{CF10}.

 While we know that $K\mapsto \mu_K$ is a Borel measure with density $\phi_K$ that depends on $K$ (by hypothesis), the ideal situation would be that $K \mapsto \mu_K(K)$ is \textit{also} a Borel measure with density (applied to, but independent of, $K$), and thus we could apply the machinery from the weighted Brunn-Minkowski theory mentioned above to the study of $\mathcal{F}$. However, this is unfortunately not the case in general. 
 
 Next, we show that the definition $\Md{K}=\mu_K(K)$ is continuous as a functional on convex bodies. Indeed, suppose $K_i\to K$ in the Hausdorff metric. Then, by definition there exists $\mu_i,\mu\in\Lambda$ with locally integrable densities $\phi_i,\phi$ respectively such that $\phi_i\to\phi$ weakly and $\Md{K_i}=\mu_i(K_i).$ Notice that
\begin{align*}\bigg|\Md{K_i}-\Md{K}\bigg| &= \bigg|\int_{K_i}\phi_i(x)dx-\int_{K}\phi(x)dx\bigg| \leq
\\
&\bigg|\int_{K}\phi_i(x)dx-\int_{K}\phi(x)dx\bigg|+\bigg|\int_{K_i}\phi_i(x)dx-\int_{K}\phi_i(x)dx\bigg|.\end{align*}
Fix an arbitrary $\epsilon>0$. The first term is bounded by $\|\phi_i-\phi\|_{L^1(K)}$ which goes to zero (say, $i$ is large enough to that $\|\phi_i-\phi\|_{L^1(K)} \leq \epsilon/3$). For the second term, one has that
\begin{align*}\bigg|\int_{K_i}\phi_i(x)dx-\int_{K}\phi_i(x)dx\bigg| &\leq \bigg|\int_{K_i\triangle K}\phi_i(x)dx\bigg| 
\\
&\leq \int_{K_i\triangle K}|\phi_i(x)|dx 
\\
&\leq \int_{K_i\triangle K}|\phi(x)|dx + \epsilon/3 
\leq \frac{2}{3}\epsilon\end{align*}
for $i$ large enough, and consequently, $\Md{K_i}\to \Md{K}.$ 

We would like to assert that the functional $\mathcal{F}$ has the form
\begin{equation}
    \mathcal{F}(K)=\frac{1}{|\alpha|}\int_{\s^{n-1}}h_K(u)dS^{\mu_K}_K(u),
    \label{eq_had_push}
\end{equation}
where $\alpha\in\R\setminus\{0\}$ is independent of $K$. We see that torsional rigidity $\tau$, via \eqref{eq:had_tor} satisfies \eqref{eq_had_push} after performing a change of variable through the Gauss map and setting $\alpha=(n+2)$. From \eqref{eq_had_push}, we then define the variational measure of $K$ with respect to the functional $\mathcal{F}$, denoted $V_{\mathcal{F},K}$, as
\begin{equation}
    dV_{\mathcal{F},K}(u)=\frac{1}{|\alpha|}h_K(u)dS^{\mu_K}_K(u).
    \label{eq:cone}
\end{equation}
The \textit{even logarithmic Minkowski problem of the functional $\mathcal{F}$} is answering the following question: Given an even Borel measure $\nu$ (with some reasonable restrictions) and a functional $\mathcal{F}$, does there exist a symmetric convex body $K$ such that $d\nu(u)=dV_{\mathcal{F},K}(u)$? In the next subsection, we require some common-sense properties on $\mathcal{F}$ to guarantee the formula \eqref{eq_had_push} holds.

\subsection{Properties of $\mathcal{F}$} Recall that we defined a diagonal set-valued Borel measure as $\Md{K}=\mu_K(K)$, where $\mu_K$ is the map given by \eqref{eq:set_dep_meas}, and we have shown this definition is continuous. In addition to continuity, we say $\mathcal{F}$ is $\alpha$-homogeneous, $\alpha \in\R\setminus\{0\}$, if $\Md{tK}=t^{\alpha}\Md{K}$ for $t>0.$ The trivial example is the measure case, i.e., when $\Md{K}=\mu(K)$ for some fixed $\alpha$-homogeneous measure $\mu.$ A more interesting example is torsional rigidity, since $\tau(tK)=t^{n+2}\tau(K)$ for $t>0$. That is, torsional rigidity is a $(n+2)$-homogeneous, diagonal set-valued Borel measure. We prove our results for $\alpha$-homogeneous, diagonal set-valued Borel measures for $\alpha \neq 0$. We will need three additional properties. Appropriate justifications will follow.

\begin{itemize}
    \item A diagonal set-valued Borel measure $\mathcal{F}$ is \textit{monotonic}, or has property \eqref{A}, if
\begin{equation}\label{A}
    K \subseteq L \Rightarrow \text{sgn}(\alpha) \Md{K} \leq \text{sgn}(\alpha)\Md{L}.
    \tag{$\mathbf{A}$}
\end{equation}

    \item  We say that an $\alpha$-homogeneous diagonal set-valued Borel measure $\mathcal{F}$ has property \eqref{B} if there exists $C>0$ such that, for any $K\in\conbod$,
\begin{equation*}\label{B}
\Md{K}^\frac{1}{\alpha} \leq C\vol_n(K)^\frac{1}{n}.\tag{$\mathbf{B}$}
\end{equation*}
\end{itemize}
Concerning \eqref{B}, recall that torsional rigidity satisfies the following isoperimetric-type inequality, known as Saint-Venant inequality (see, e.g., \cite{PS51}):
\begin{equation}
    \left(\frac{\tau(K)}{\tau(B_2^n)}\right)^\frac{1}{n+2}\leq \left(\frac{\vol_n(K)}{\vol_n(B_2^n)}\right)^{\frac{1}{n}}.
    \label{eq_dSV}
\end{equation}
In general, isoperimetric inequalities guarantee suitable compactness for the variational approach to the solution of Minkowski problems. Thus, it is reasonable to require \eqref{B} to control $\Md{K}$ by $\vol_n(K)$. Property \eqref{A} plays a similar role: The two inequalities together provide a lower and an upper bound for $\Md{K}$. They exchange roles depending on $\text{sgn}(\alpha)$, and when their interaction does not follow \eqref{A} and \eqref{B}, our strategy fails. This is the case, for example, for the capacity functional. See Section \ref{sec:remarks_on_capacity} for more details.

The third property is property \eqref{C}, and this is a bit more delicate. It is by this property that $\mathcal{F}$ has the representation \eqref{eq_had_push}. The following facts are reported, for example, in \cite{Sh1}. For every strictly positive $f\in C(\s^{n-1})$, the \textit{Wulff shape} of $f$ is the convex body given by
	\begin{equation}
	    [f]=\{x\in\R^n:\langle x,u\rangle\leq f(u) \; \text{for all}\; u  \in \s^{n-1}\}.
	\end{equation}
	One has that, for a convex body $K$ containing the origin in its interior, $[h_K]=K.$ Since $f$ is positive, $[f]$ is a convex body containing the origin in its interior. Furthermore, if $f$ is even, then $[f]$ is symmetric. Next, for $f\in C(\s^{n-1})$, Aleksandrov \cite{AL} defined a perturbation of $K$ to be the Wulff shape of the function
	$$h_t(u)=h_K(u)+tf(u),$$
	where $t\in (-\delta,\delta)$, $\delta$ small enough so that $h_t$ is positive for all $u$. From here, Aleksandrov showed his variational formula:
	\begin{equation}
 \diff{\vol_n([h_{t}])}{t}\bigg|_{t=0}=\lim_{t\to 0}\frac{\vol_n([h_t])-\vol_n(K)}{t}=\int_{\s^{n-1}}f(N_K(x))d\mathcal{H}^{n-1}(x).
\label{eq:vari_vol}
        \end{equation}
It is natural to ask if other functionals besides volume have a variational formula of type \eqref{eq:vari_vol}. Indeed, Colesanti and Fimiani \cite[Theorem 4.1]{CF10} showed, for $K\in\conbod$ and $f\in C(\s^{n-1})$, that
$$ \lim_{t\to 0}\frac{\tau([h_K+tf])-\tau(K)}{t}=\int_{\partial K}f(N_K(x))|\nabla u_K(x)|^2d\mathcal{H}^{n-1}(x),$$
where $u_K$ solves \eqref{eq:tor_PDE}. With more generality, for an $\alpha$-homogeneous diagonal set-valued Borel measure $\mathcal{F}$ we want to require for every $K\in\conbod$ and $f\in C(\s^{n-1})$, there exists a non-negative, locally integrable function $\phi_K$ on $\partial K$ such that
\begin{equation}
\label{eq:og_C}
    \frac{d}{dt}\mathcal{F}([h_K+tf])\Big|_{t=0}=\text{sgn}(\alpha)\int_{\partial K} f(N_K(x)) \phi_K(x) d\mathcal{H}^{n-1}(x).
\end{equation}
Notice by using the pushforward of $N_K$, we can write
$$\frac{d}{dt}\mathcal{F}([h_K+tf])\Big|_{t=0}=\text{sgn}(\alpha)\int_{\s^{n-1}} f(u) dS_K^{\mu_K}(u).$$ For our purposes, we will consider not $h_K+tf$ but $h_Ke^{tf}$, which is a logarithmic perturbation of $h_K$ by $f$, see, e.g., \cite{BLYZ12,HLYZ16}. Notice, however, that $$h_Ke^{tf}=h_K + tfh_K + o(t^2) \quad \text{and} \quad h_K+tf = h_Ke^{t\frac{f}{h_K}}+o(t^2).$$  Additionally, recall the classical fact that if $f_i\to f$ with respect to the sup-norm on $C(\s^{n-1}),$ then $[f_i]\to [f]$ in the Hausdorff metric. Consequently, the operator $\Md{[\cdot]}$ is a continuous functional on $C(\s^{n-1})$. Therefore, a variational formula for a perturbation of the form $h_K+tf$ is equivalent to a variational formula for a perturbation of the form $h_Ke^{tf}$. We are thus justified to require $\eqref{C}$, which is given as follows.

\begin{itemize}
\item An $\alpha$-homogeneous diagonal set-valued Borel measure $\mathcal{F}$ is said to have \textit{Hadamard derivative}, or has property \eqref{C}, if for every $K\in\conbod$ and $f\in C(\s^{n-1})$, there exists a Borel measure $\mu_K$ with a non-negative, locally integrable density $\phi_K$ on $\partial K$ such that, by setting, $K_t=[h_Ke^{tf}]$,
\begin{equation}
    \lim_{t\to 0}\frac{\Md{K_t}-\Md{K}}{t}=\text{sgn}(\alpha)\int_{\s^{n-1}}h_K(u)f(u)dS^{\mu_K}_{K}(u).
    \label{C}
    \tag{$\mathbf{C}$}
\end{equation}
\end{itemize}
\noindent Again, $\eqref{C}$ is actually equivalent to \eqref{eq:og_C}. It is now trivial that property $\eqref{C}$ together with homogeneity yields $\mathcal{F}$ has a representation of the form \eqref{eq_had_push}. Indeed, for every $t>0, \Md{(1+t)K}=(1+t)^\alpha\Md{K}$. Taking a derivative and evaluating at $t=0$ yields, by using \eqref{eq:og_C},
$$\Md{K} = \frac{1}{\alpha}\frac{d}{dt}\Md{(1+t)K}\Big|_{t=0}=\frac{1}{|\alpha|}\int_{\partial K} h_K(N_K(x)) \phi_K(x) d\mathcal{H}^{n-1}(x),$$
and then the representation \eqref{eq_had_push} follows by using the Gauss map. Now that we have introduced diagonal set-valued Borel measures, we are ready to state our main theorem.

\begin{theorem}
\label{mainTh}
    Let $\nu$ be an even, finite, and positive Borel measure over $\s^{n-1}$ such that it satisfies the strict subspace concentration inequality \eqref{eq:SSCC}. Fix $\alpha\neq 0$. Let $\mathcal{F}$ be a set-dependent, $\alpha$-homogeneous Borel measure satisfying properties $\eqref{A},$ $\eqref{B}$, and $\eqref{C}$. Let $V_{\mathcal{F},K}$ be the variational measure of $K$ with respect to the functional $\mathcal{F}$ given by \eqref{eq:cone}. Then, there exists a $K\in\conbod_e$ such that
    $d\nu(u)=dV_{\mathcal{F},K}.$ In particular, one has $$\nu(\s^{n-1})=\mathcal{F}(K).$$
\end{theorem}

\subsection{The Eigenvalue of the Laplacian}
\label{sec:diagonal}
We emphasized torsional rigidity in the introduction when discussing diagonal set-valued Borel measures. We now provide a further example: The first eigenvalue of the Laplacian. For $K\in\conbod,$ consider the following eigenvalue problem:

\begin{equation}
\label{eq:PDE_eigen}
\begin{cases}-\Delta v(x) = \lambda_1(K) v(x) & \text { for } x \in \text{int}(K), \\ \;\quad v(x)=0 & \text { for } x \in \partial K.\end{cases}
\end{equation}
Being a linear PDE, any multiple of a solution is a solution. We will consider the unique solution to \eqref{eq:PDE_eigen} that satisfies $\int_{K} v^2 dx =1$; we denote this solution as $v_K$. Here, the constant $\lambda_1(K)$ is the smallest non-trivial eigenvalue of the Laplacian and is known as the \textit{principal eigenvalue} of $K$. The principal eigenvalue is $(-2)$-homogeneous. One can easily verify that $\lambda_1$ is monotonically decreasing, i.e., satisfies property $\eqref{A}.$
It also satisfies the \textit{Faber-Krahn inequality}
$$\frac{\lambda_1(B_2^n)}{\lambda_1(K)}\leq \left(\frac{\vol_n(K)}{\vol_n(B_2^n)}\right)^{2/n}.$$
%Notice that this can be written as
%$$\left(\frac{\lambda_1(K)}{\lambda_1(B_2^n)}\right)^{\frac{1}{-2}} \leq \left(\frac{\vol_n(K)}{\vol_n(B_2^n)}\right)^\frac{1}{n},$$
Therefore, the principal eigenvalue also has property $\eqref{B}.$ 
Jerison showed \cite[Theorem 7.5]{DJ96_2} that $\lambda_1(K)$ has the following formula.
\begin{equation}
\begin{split}
    \lambda_1(K)&=\frac{1}{2} \int_{\partial K} h_{K}(N_K(x))|\nabla v_K(x)|^2d\mathcal{H}^{n-1}(x) 
    \\
    &=\frac{1}{2}\int_{\s^{n-1}}h_K(u) dS^{\mu_K}_K (u),
    \label{eigen_form}
    \end{split}
\end{equation}
where $S^{\mu_K}_K$ is a Borel measure on $\s^{n-1}$ given by the pushforward of $|\nabla v_K(x)|^2$ from $\partial K$ to $\s^{n-1}$ via the Gauss map (see \eqref{eq:surface_mu_2}).

Thus, \eqref{eigen_form} and the above discussion shows that the principal eigenvalue is a $(-2)-$homogeneous, set-dependent Borel measure satisfying properties $\eqref{A},\eqref{B},$ and $\eqref{C}$. The variational measure in this instance is then
\begin{equation} 
\label{eq:cone_lambda}
dV_{\lambda_1,K}(u)\coloneqq\frac{1}{2}h_K(u)dS_K^{\mu_K}(u).\end{equation}
Notice that $V_{\lambda_1,K}(\s^{n-1})=\lambda_1(K).$ Thus, in this instance, the variational measure will be called the Poincar\'e measure. Our main theorem below, Theoerm~\ref{mainTh}, therefore implies the following. 
\begin{theorem}
\label{t:first_eigen}
    Let $\nu$ be an even, finite Borel measure on $\s^{n-1}$ such that $\nu$ satisfies the subspace concentration condition. Then, there exists an origin symmetric convex body $K\in\conbod_e$ such that $\nu$ is the Poincar\'e measure of $K$, i.e. $d\nu(u)=dV_{\lambda_1,K}(u)$. In particular, one has
    $$\nu\left(\s^{n-1}\right)=\lambda_1\left(K\right).$$
\end{theorem}

\section{All Down the Line: The Main result}
\label{sec:main}
In this section, we prove Theorem~\ref{mainTh}. We will follow the scheme from \cite{BLYZ13}. First, we consider the following minimization problem. Let $\nu$ be a finite even Borel measure on $\s^{n-1}$ with total mass $|\nu|>0$. Define the functional $\Phi_\nu: \conbod_e \rightarrow \R$ given by
\begin{equation}
    \Phi_\nu(K)=\int_{\s^{n-1}} \log h_K(u) d \nu(u).
    \label{eq:func}
\end{equation} 
We next will consider the following minimization problem and show that its solution is the variational measure of the functional $\mathcal{F}$:
\begin{equation}
    \inf_{Q\in\conbod_e} \left\{\Phi_\nu(Q): \mathcal{F}(Q)=|\nu|\right\}.
    \label{eq:min}
\end{equation}
It will be convenient to introduce the notation $C_e^{+}(\s^{n-1})$ for the set of all positive, even and continuous functions on the sphere.
\begin{lemma}
\label{main_lemma}
    Let $\nu$ be a finite, even Borel measure on $\s^{n-1}$ such that $|\nu|>0.$ Fix $\alpha\neq 0$. Let $\mathcal{F}$ be an $\alpha$-homogeneous set-dependent Borel measure satisfying property $\eqref{C}$ for sets in $\conbod_e$. Then, if $K_0\in\conbod_e$ is symmetric such that $\Md{K_0}=|\nu|$ and
    \begin{equation}
         \Phi_{\nu}(K_0)=\inf_{Q\in\conbod_e} \left\{\Phi_\nu(Q): \mathcal{F}(Q)=|\nu|\right\},
    \end{equation}
    then $\nu$ is the variational measure for $K_0$ associated with $\mathcal{F}$ given by \eqref{eq:cone}.
\end{lemma}
\begin{proof}
Via the homogeneity of $\Md{\cdot},$ we may assume that $\nu$ is a probability measure on the sphere. For $q\in C_e^{+}(\s^{n-1})$, define $$\Gamma(q)\coloneqq\frac{1}{\Md{[q]}^{1/\alpha}}\exp\left(\int_{\s^{n-1}}\log(q)d\nu\right).$$
From the definition of $\mathcal{F}$, $\Md{[q]}$ a continuous functional on $C_e^{+}(\s^{n-1})$. Furthermore, we see that $\Gamma(q)$ is homogeneous of degree 0, i.e. $\Gamma(sq)=\Gamma(q)$ for all $s>0.$ Next, consider the minimization problem 
\begin{equation}
    \inf \left\{\Gamma(q): q \in C_e^{+}\left(\s^{n-1}\right)\right\}.
    \label{eq:min_2}
\end{equation}

We first show that the solution to this minimization problem is obtained among support functions of symmetric convex bodies. Indeed, for $q\in C_e^{+}\left(\s^{n-1}\right),$ one has $\Md{[q]}=\Md{[h_{[q]}]}$ and yet $h_{[q]} \leq q$ point-wise. Therefore, $\Gamma(h_{[q]}) \leq \Gamma(q).$ From the fact that $\Gamma$ is homogeneous of degree 0 and that we can restrict our search to support functions of origin symmetric convex bodies, we obtain that
$$\inf \left\{\Gamma(q): q \in C_e^{+}\left(\s^{n-1}\right)\right\}=\inf_{Q\in\conbod_e} \left\{e^{\Phi_\nu(Q)}: \mathcal{F}(Q)=|\nu|\right\}.$$
The infimum on the right-hand side is obtained at $h_{K_0}$ by hypothesis. Consequently, we obtain that
$$\inf \left\{\Gamma(q): q \in C_e^{+}\left(\s^{n-1}\right)\right\}=\Gamma(h_{K_0}).$$
Next, fix some arbitrary even and continuous function $g$ on $\s^{n-1}$. Define the family $h_t=h_{K_0}e^{tg},$ and let $K_t=[h_t].$ Then, via property $\eqref{C},$ one has
$$\diff{\Md{K_t}}{t}\bigg|_{t=0}=\text{sgn}(\alpha)\int_{\s^{n-1}}gh_{K_0}dS^{\mu_{0}}_{K_0},$$
where $\Md{K_0}=\mu_0(K_0).$ Therefore, we obtain that
$$\Gamma\left(h_t\right)=\Md{K_t}^{-1 / \alpha} \exp \left(\int_{\s^{n-1}} \log \left(h_{K_0} e^{t g}\right) d \nu\right)$$
is differentiable at $t=0.$ From the chain rule and $\eqref{C}$, we obtain
\begin{equation}
\diff {\Gamma\left(h_t\right)}{t}\bigg|_{t=0}=\exp \left(\int_{\s^{n-1}} \log h_{K_0} d \nu\right)\left[-\frac{1}{|\alpha|} \int_{\s^{n-1}} g h_{K_0} d S^{\mu_0}_{K_0}+\int_{\s^{n-1}} g d \nu\right].
\label{eq:M_deriv}
\end{equation}
On the other hand, the fact that $\Gamma(h_{K_0})$ minimizes \eqref{eq:min_2} shows that
$$\diff{\Md{K_t}}{t}\bigg|_{t=0}=0.$$
Combining this fact with \eqref{eq:M_deriv} we obtain, since $g$ is arbitrary,
$$\frac{1}{|\alpha|}h_{K_0} d S^{\mu_0}_{K_0}=d \nu,$$
as desired.
\end{proof}
\noindent We are now ready to prove the main theorem by showing there exists a symmetric convex body $K_0$ solving \eqref{eq:min}.
\begin{proof}[Proof of Theorem~\ref{mainTh}]
We first show that there exists a symmetric convex body $K\in\conbod_e$ such that
\begin{equation}
\int_{\s^{n-1}} \log h_K(u) d \nu(u)=\inf_{Q\in\conbod_e} \left\{\int_{\s^{n-1}} \log h_Q(u) d \nu(u): \mathcal{F}(Q)=|\nu|\right\}.
\label{solving}
\end{equation}

From the homogeneity of $\mathcal{F}(\cdot),$ we can again assume that $\nu$ is a probability measure. Consider a sequence $\{Q_l\}\subset\conbod_e$ such that $\Md{Q_l}=1$ and 
$$\lim_{l\to\infty}\Phi_{\nu}(Q_l)=\inf_{Q\in\conbod_e} \left\{\Phi_\nu(Q): \mathcal{F}(Q)=1\right\}.$$

\noindent Let $m_n\coloneqq\Md{B_2^n}$, and set $B_m=m_n^{-1/\alpha}B_2^n$ so that $\Md{B_m}=1.$ Notice that 
\begin{equation}\Phi_{\nu}(B_m)=-\frac{1}{\alpha}\log{m_n}.
\label{eq:bound}
\end{equation}
Consequently,
$$\lim_{\ell\to\infty}\Phi_{\nu}(Q_\ell) \leq -\frac{1}{\alpha}\log{m_n}.$$
Following the approach from \cite[Theorem 6.3]{BLYZ13}, since each $Q_l$ is non-empty, there exists cross-polytopes (via John's theorem) $C_l$ such that $$C_l \subset Q_l \subset n C_l, \quad C_l=\left[\pm h_{1, l} u_{1, l}, \ldots, \pm h_{n, l} u_{n, l}\right]$$ for some set $\{u_{i,l}\}_{i=1}^n\subset \s^{n-1},$ where $h_{i, l}=h_{C_l}\left(u_{i, l}\right).$
Furthermore, the indices are indexed so that
$h_{1, l} \leq \cdots \leq h_{n, l}.$ By way of contradiction, suppose the sequence $\{Q_l\}$ is not bounded. 
Then, the sequence $\{C_l\}$ is not bounded. Therefore, by passing to a subsequence if need be, one has
$$\lim_{l\to\infty}h_{n,l}=\infty.$$

On the other hand, since $\Md{Q_l}=1$ and $\Md{\cdot}$ satisfies property $\eqref{A}$ one has that $n^{-\alpha} \leq \Md{C_l} \leq 1$ if $\text{sgn}(\alpha) >0$ and $1 \leq \Md{C_l} \leq n^{-\alpha}$ if $\text{sgn}(\alpha) <0$. In either case, $\Md{\cdot}$ satisfying property $\eqref{B}$ implies there exists a sequence of numbers $A_l$ bounded uniformly away from $0$ and $\infty$ such that $\vol_n(C_l) \geq A_l.$ Then, from the formula of the volume of a cross-polytope, we obtain that
\begin{equation}
\label{supp_prod}
\prod_{i=1}^nh_{i,l}=\frac{n!\vol_n(C_l)}{2^n} \geq \frac{n!A_l}{2^n}\end{equation}
Notice that, with $\widetilde{C_l}=\left(\frac{n!A_l}{2^n}\right)^{-\frac{1}{n}}C_l$, $$\Phi_\nu\left(\widetilde{C_l}\right)=\int_{\s^{n-1}} \log h_{\widetilde{C_l}}(u) d \nu(u)=\frac{1}{n}\log\left(\frac{2^n}{n!A_l}\right)+\Phi_\nu(C_l).$$

 One then obtains from \eqref{supp_prod} and \cite[Lemma 6.2]{BLYZ13} that $\{\Phi_{\nu}(\widetilde{C}_l)\}$ is not bounded from above. But this implies $\{\Phi_{\nu}(Q_l)\}$ is not bounded from above, which contradicts \eqref{eq:bound} for $l$ large enough. Thus, we must have that $\{Q_l\}$ is bounded. From the Blaschke Selection Theorem \cite{Sh1}, $\{Q_l\}$ has a subsequence which converges to an origin symmetric convex body $K\in\conbod_e$, and by construction, this $K$ solves \eqref{solving}. Then, from Lemma~\ref{main_lemma}, it solves our claim.
\end{proof}

\section{Variational Rolling Stones}
\label{sec:worn}
Throughout this section, we will assume all convex bodies are $C^2_+$. Following Tso \cite{KT85}, one can consider a weighted version of \eqref{eq:WornPDE}:  For $\xi\in\s^{n-1}$ and $t\in [0,T)$ with a fixed $T>0$
\begin{equation}
\label{eq:WornPDE_2}
    \pdv{h(t,\xi)}{t}=-T\varphi(\xi)\kappa(t,\xi)
\end{equation}
for some continuous, positive function $\varphi(\xi)$. We recall that $h(t,\xi)=h_{K(t)}(\xi)$ for some collection of $C^2_+$ convex bodies $\{K(t)\}$. Repeating the above framework discussed in the introduction for the case when $\varphi$ is a positive constant, self-similar solutions (which we recall means $h(t,\xi)=T^{-\frac{1}{n}}h_{K(0)}(\xi)(T-t)^{\frac{1}{n}}$) satisfy 
\[
\frac{1}{n}h_{K(0)}(\xi)  dS_{K(0)}(\xi)=\varphi(\xi)d\xi,
\]
which is again the log-Minkowski problem, this time for the Borel measure on the sphere with density $\varphi(\xi)$. One again obtains, when $\varphi$ is an even function, that self-similar solutions exist via \cite{BLYZ13}.

We now explain the variational version of this framework, focusing on the torsional rigidity case. The reader can deduce a similar outline for a more generic $\alpha$-homogeneous, diagonal set-valued Borel measure. Recall that given a convex body $K$, there is a unique solution $u_K$ solving \eqref{eq:tor_PDE}. Furthermore, the torsional rigidity $\tau$ can be viewed as a diagonal set-valued Borel measure with density $|\nabla u_K(x)|^2$. For $\xi\in\s^{n-1}$ and $t\in [0,T)$ with a fixed $T>0$, consider
\begin{equation}
\label{eq:VariPDE}
    |\nabla u_{K(t)}(N^{-1}_{K(t)}(\xi))|^2\pdv{h(t,\xi)}{t}=-T^{-1}\varphi(\xi)\kappa(t,\xi).
\end{equation}
We then obtain  the following Monge-Amp\`ere equation:
\[
 |\nabla u_{K(t)}(N^{-1}_{K(t)}(\xi))|^2\pdv{h(t,\xi)}{t} \det(D^2h(t,\xi)+h(t,\xi)\text{I}_{n-1})=-T^{-1}\varphi(\xi).
\]
We recall that the solution to the torsional rigidity problem satisfies the following ``pseudo"-homogeneity:
$$u_{cK}(cx)=c^{2}u_{K}\left(x\right).$$
for $c>0$. This then implies that
$$|\nabla u_{tK}(cx)|^2=c^2|\nabla u_K(x)|^2.$$
We now consider self-similar solutions of the form $$h(t,\xi)=T^{-\frac{1}{n+2}}h_{K(0)}(\xi)(T-t)^{\frac{1}{n+2}}.$$
Then, \begin{align*}\pdv{h(t,\xi)}{t}&= -T^{-\frac{1}{n+2}}\frac{1}{n+2}(T-t)^{-\frac{n+1}{n+2}}\quad\text{and } 
\\
|\nabla u_{K(t)}(N^{-1}_{K(t)}(\xi))|^2 &= T^{-\frac{2}{n+2}}(T-t)^{\frac{2}{n+2}}|\nabla u_{K}(N^{-1}_{K}(\xi))|^2.
\end{align*}
Using the fact that $$|\nabla u_{K}(N^{-1}_{K}(\xi))|^2 dS_{K(0)}(\xi) = S_{K(0)}^{\mu_0}(\xi),$$
where $\mu_0$ is the measure with density $|\nabla u_{K_0}|^2$, we obtain

\[
\frac{1}{n+2}h_{K(0)}(\xi)  dS_{K(0)}^{\mu_0}(\xi)=\varphi(\xi)d\xi,
\]
which is, in the case $\varphi$ is even, the log-Minkowski problem shown in Theorem~\ref{t:tor_prob} with a Borel measure $\nu$ on $\s^{n-1}$ that has density $\varphi$. We collect this observation in the following corollary of Theorem~\ref{t:tor_prob}.
\begin{corollary}
For $t\in (0,T]$, consider the curvature flow given by \eqref{eq:VariPDE}, where $u_{K(t)}$ is the solution to the torsional rigidity problem \eqref{t:tor_prob} on the convex body $K(t)$, $\{K(t)\}$ is a collection of $C^2_+$ convex bodies indexed by $t$, and $\varphi$ is a continuous, even, positive function on $\s^{n-1}$. Then, there exists a self-similar solution with death time $T$ to this problem. That is, there is a symmetric convex body $K(0)$ such that, for every $t\in [0,T)$ and $\xi\in\s^{n-1}$,
$$h_{K(t)}(\xi)=T^{-\frac{1}{n+2}}h_{K(0)}(\xi)(T-t)^{\frac{1}{n+2}},$$
and
$$\frac{1}{n+2}h_{K(0)}(\xi)  |\nabla u_{K}(N^{-1}_{K}(\xi))|^2 dS_{K(0)}(\xi)=\varphi(\xi)d\xi.$$
\end{corollary}
We note that the curvature flow problem \eqref{eq:VariPDE} was first introduced by Crasta and Fragal\'a \cite{CF23}. However, in place of $\varphi(\xi)$, they considered $\tau(K(t))$. Considering this slightly different problem, they showed that \textit{if} a solution exists, then it converges to a ball along the corresponding curvature flow. However, from the homogeneity of $\tau$, a self-similar solution to their version of \eqref{eq:VariPDE} of the form $h_{K(t)}(\xi)=T^{-\beta}h_{K(0)}(\xi)(T-t)^\beta$ cannot exist (via direct substitution).

\section{Remarks on Capacity}\label{sec:remarks_on_capacity}
One of the most common functionals in the calculus of variations is capacity. The reader may wonder why we have not mentioned such an important functional in the previous sections. Thus, we conclude with some remarks concerning the capacity functional and its generalizations; we show that capacity falls outside the framework introduced in Section~\ref{sec:gen}. Suppose that $\Omega$ is a bounded domain, and let  $\Delta_p$ denote the $p$-Laplacian. Then, for $p>1$ we consider the following system:
 \begin{equation}\label{eq:capacity}
 \begin{cases}
     \Delta_p \omega (x)=0 \; &\text{for } x \in \R^n\setminus\Omega,
     \\
     \omega(x) = (p-1)^\frac{1}{p} \; &\text{for } x \in \partial \Omega,
     \\
     \lim_{|x|\to\infty}\omega(x)=0.
\end{cases}
 \end{equation}
  The (unique) solution to \eqref{eq:capacity} is called the \textit{p-capacitary function}, which we denote as $\omega_\Omega$ (here, we suppress the dependence on $p$). The $p$-capacitary function generates the $p$-capacity:
  $$C_p(\Omega)=\frac{1}{p-1}\int_{\R^n\setminus{\overline{\Omega}}}|\nabla \omega_\Omega(x)|^pdx.$$
  We remark that we slightly changed the usual PDE associated to $p$-capacity; usually, in \eqref{eq:capacity}, it is required that $\omega(x)=1$ on $\partial \Omega$ (for all $p$), and the solution of the (usual) PDE is merely $(p-1)^{-\frac{1}{p}}$ times our solution. The quantity $C_p(\Omega)$ is still the same, i.e.
  $$C_p(\Omega)=\inf \left\{\int_{\R^n}|\nabla u|^p d x: u \in C_c^{\infty}\left(\R^n\right) \text { and } u \geq 1 \text { on } \Omega\right\},$$
  where $C_C^\infty(\R^n)$ is the set of infinitely differentiable, compactly support functions on $\R^n$. This slightly different choice for the PDE allows for the presentation of facts below.
  
  For a convex body $K$ (again writing $K$ for int$(K)$ in an abuse of notation), the $p$-capacity has the following Hadamard derivative (known as the Poincar\'e formula): For $1<p<n$ \cite{CNSXYZ15} and $f\in C(\s^{n-1})$
\begin{equation*}
\begin{split}
     C_p(K)&=\frac{1}{n-p}\int_{\s^{n-1}}h_{K}(u)d\mu_p(K,u) \quad \text{and}
     \\
     &\quad\quad \frac{d}{dt}C_p([h_K+tf])\Big|_{t=0}= \int_{\s^{n-1}} f(u)d\mu_p(K,u),
     \end{split}
 \end{equation*}
 where $\mu_p(K,u)$ is the $p$-capacitary measure (which is the pushforward of $|\nabla \omega_K|^p$ from $\partial \Omega$ to $\s^{n-1}$ \cite{DJ96_2, CNSXYZ15, CS03, CJL96}, i.e., $\mu_p(K,\cdot)$ corresponds to $S^{\mu_K}_K$ from Section~\ref{sec:set_dep}). It is easy to verify that $p$-capacity is monotonically increasing and $(n-p)$-homogeneous.
 
 When $p=2$, the $p$-capacity is called the Newton capacity; Jerison solved the Minkowski problem in the Newton capacity case \cite{DJ96} and Colesanti et al. \cite{CNSXYZ15} settled the Minkowski problem for $p$-capacity. Unfortunately, our procedure does not yield the existence of solutions to the even logarithmic Minkowski problem for $p$-capacity, as the isoperimetric inequality in this case, the \textit{Szeg\"o inequality}, goes the wrong way: For every $p\geq 1,$
$$\left(\frac{\vol_n(K)}{\vol_n(B_2^n)}\right)^\frac{1}{n} \leq \left(\frac{C_p(K)}{C_p(B_2^n)}\right)^\frac{1}{n-p},$$
and this is impossible to reverse since there exist sets with positive capacity but zero volume. In summary, $C_p$ satisfies properties $\eqref{A}$ and $\eqref{C}$ but not $\eqref{B}$.  \\Recently, the Logarithmic Minkowski problem for the capacity was settled in the discrete case \cite{XGXJ22}. Thus, the necessity of a suitable isoperimetric inequality might be an inherent vice of the method at hand, opening interesting developments on how to overcome this problem.

\bigskip
{\bf Acknowledgments:} We extend heartfelt thanks to Gabriele Bianchi, Graziano Crasta, Matthieu Fradelizi, Ilaria Fragal\'a, Paolo Gronchi, David Jerison, and Artem Zvavitch for the helpful comments concerning this work. We also thank Yiming Zhao, who suggested the problem at the 2022 ``Workshop in Convexity and High-Dimensional Probability", organized by Galyna Livshyts and hosted at Georgia Tech University. We are grateful to the reviewer for the thorough and insightful comments, which significantly improved the presentation of this work.

\printbibliography

 \begin{tabular}{l l}
Dylan Langharst & Jacopo Ulivelli\\
Institut de Math\'ematiques de Jussieu & Diparitimento di Matematica Guido Castelnuovo\\
Sorbonne Universit\'e & Sapienza, University of Rome\\
Paris, 75252 & Roma, Piazzale Aldo Moro 00185 \\
France & Italy \\
  dylan.langharst@imj-prg.fr & jacopo.ulivelli@uniroma1.it\\
\end{tabular}

\end{document}